\documentclass[reqno,12pt,a4letter]{amsart}
\usepackage{amsmath,amsxtra,amssymb,latexsym,amscd,amsthm}
\usepackage{graphicx,color}
\usepackage[active]{srcltx}
\usepackage[utf8]{inputenc}
\usepackage[mathscr]{euscript}
\usepackage{mathrsfs,cite,enumerate}
\usepackage[english]{babel}
\usepackage{color}
\usepackage[active]{srcltx}
\usepackage{comment}
\newtheorem {theorem}{Theorem}[section]

\newtheorem {lemma}{Lemma}[section]
\newtheorem {example}{Example}[section]
\newtheorem {definition}{Definition}[section]

\title{Approximate solutions of interval-valued optimization problems}

\author{NGUYEN VAN TUYEN$^{1}$}
\address{$^1$Department of Mathematics, Hanoi Pedagogical University 2, Xuan Hoa, Phuc Yen, Vinh Phuc, Vietnam}
\email{tuyensp2@yahoo.com; nguyenvantuyen83@hpu2.edu.vn}

\thanks{}

\date{\today}



\begin{document}

\maketitle

\begin{abstract}
This paper deals with approximate solutions of an optimization problem with interval-valued objective function. Four types of  approximate solution concepts of the problem are proposed by considering the partial ordering $LU$ on the set of all closed and bounded intervals.  We show that these  solutions exist under very weak conditions. Under suitable constraint qualifications, we derive Karush--Kuhn--Tucker necessary and sufficient  optimality conditions for convex interval-valued optimization problems.
\end{abstract}

\section{Introduction}
\label{intro}
In this paper, we are interested in approximate solutions of the following constrained interval-valued optimization problem: 
\begin{align}\label{problem}\tag{P}
\begin{split}
&\mathrm{min}\, f(x) 
\\
&\text{s. t.}\ \ x\in X:=\{x\in\mathbb{R}^n\,:\, g_j(x)\leq 0, j=1, \ldots, m\}, 
\end{split}
\end{align}
where $f\colon \mathbb{R}^n\to \mathcal{K}_c$ is an interval-valued function defined by $f(x)=[f^L(x), f^U(x)]$, $f^L,$ $f^U\colon\mathbb{R}^n\to \mathbb{R}$  are real-valued functions satisfying $f^L(x)\leq f^U(x)$ for all $x\in\mathbb{R}^n$, $\mathcal{K}_c$ denote the class of all closed and bounded intervals in $\mathbb{R}$, i.e., 
$$\mathcal{K}_c=\{[a^L, a^U]\,:\, a^L, a^U\in\mathbb{R}, a^L\leq a^U\},$$ 
$g_j\colon\mathbb{R}^n\to \mathbb{R}$, $j\in J:=\{1, \ldots, m\}$, are real-valued constraint functions.

The interval-valued optimization problems  recently have received increasing interest in optimization community;  see, e.g., \cite{Chalco-Cano et.al.-13,Ishibuchi-Tanaka-90,Singh-Dar-15,Singh-Dar-Kim-16,Singh-Dar-Kim-19,Tung-2019,Tung-2019b,Wu-07,Wu-09,Wu-09-b} and references therein. The reason for this is that many problems  in decision making, engineering and economics are affected by risk and uncertainty; see, e.g.,   \cite{Ben-Tal-Nemirovski-98,Ben-Tal-Nemirovski-08,Ben-Tal-Nemirovski-09,Moore-1966,Moore-1979,Slowinski-98}. Hence, we usually cannot determine exactly the coefficients of  objective functions in such problems.  If the coefficients of objective functions are taken as closed intervals, we obtain interval-valued optimization problems of the form \eqref{problem}. These problems may provide an alternative choice for considering optimization problems with uncertain or imprecise data.

In interval-valued optimization, it is important to compare intervals by means of interval order relations. There is a variety of interval order relations known in the literature; see, e.g.,   \cite{Alefeld,Chalco-Cano et.al.-13,Ishibuchi-Tanaka-90,Moore-1966,Moore-1979}.  The well know lower-upper ($LU$) interval order relation    and center-width ($CW$) one are introduced by Ishibuchi and  Tanaka \cite{Ishibuchi-Tanaka-90}.     The lower-spread ($LS$) interval order relation  was proposed by Chalco-Cano et. al. \cite{Chalco-Cano et.al.-13}. For these interval order relations, the corresponding solution concepts for the optimization problem with interval-valued objective function are introduced and studied. 

As a mainstream in the study of interval-valued optimization problems, Karush--Kuhn--Tucker (KKT)  optimality conditions for interval-valued optimization problems  have attracted the attention of many researchers; see, e.g.,  \cite{Chalco-Cano et.al.-13,Singh-Dar-15,Singh-Dar-Kim-16,Singh-Dar-Kim-19,Tung-2019,Tung-2019b,TXS-20,Wu-07,Wu-09,Wu-09-b} and  the references therein. However, to the best of our
knowledge, so far there have been no papers investigating optimality
conditions of KKT-type  for approximate solutions of interval-valued optimization problems. It should be noted  that, in general optimization problems, the study of approximate solutions is very
important because, from the computational point of view, numerical algorithms usually generate only approximate solutions because they stop after a finite number of steps. Furthermore, approximate solutions exist under very weak assumptions; see, e.g.,  \cite{Bao-et al,Loridan-84,Son-Tuyen-Wen-19,Tammer,Tammer-Zalinescu}.

In this paper, we focus for the first time on studying the existence and optimality conditions of KKT-type for approximate solutions of interval-valued optimization problems. We first introduce in the next  section  four kinds of approximate solutions with respect to $LU$ interval order relation of \eqref{problem}. Then we show that the new concepts of approximate solutions are closed related to the approximate efficient solutions of multiobjective optimization problems in the sense of Loridan \cite{Loridan-84}. Section \ref{Existence theorems} is devoted to study the existence of proposed approximate solutions.  In Section \ref{KKT conditions}, we establish KKT necessary and sufficient  optimality conditions  for approximate solutions to  convex interval-valued optimization problems of the form \eqref{problem} under suitable constraint qualifications.   
\section{APPROXIMATE SOLUTIONS} \label{Approximate solutions}
We use the following notation and terminology. Fix $n \in {\mathbb{N}}:=\{1, 2, \ldots\}$. The space $\mathbb{R}^n$ is equipped with the usual scalar product and  Euclidean norm.  We denote the nonnegative orthant in $\mathbb{R}^n$ by  $\mathbb{R}^n_+$.  The topological closure, the topological interior  and the convex hull  of a subset  $S$ of $\mathbb{R}^n$ are denoted, respectively, by  $\mathrm{cl}\,{A}$,  $\mathrm{int}\,{A}$  and $\mathrm{conv}\,A$. The conical hull    of $A$ is defined by
\begin{equation*}
\mathrm{cone}\,A:=\{\lambda x\,:\, \lambda\geq 0, x\in \mathrm{conv}\,A\}.
\end{equation*}

Let $A=[a^L, a^U]$ and $B=[b^L, b^U]$ be two intervals in $\mathcal{K}_c$. Then, by definition, we have
\begin{enumerate}[(i)]
	\item $A+B=\{a+b\,:\, a\in A, b\in B\}=[a^L+b^L, a^U+b^U]$;
	\item $A-B=\{a-b\,:\, a\in A, b\in B\}=[a^L-b^U, a^U-b^L]$.
\end{enumerate}
We also see that  
\begin{equation*}kA=\{ka\,:\, a\in A\}=
\begin{cases}
[ka^L, ka^U]\ \ \text{if}\ \ k\geq 0,
\\
[ka^U, ka^L]\ \ \text{if}\ \ k < 0, 
\end{cases}
\end{equation*}
where $k$ is a real number, see \cite{Alefeld,Moore-1966,Moore-1979}  for more details.

Let $A\subset\mathbb{R}^n$ and $B\subset\mathbb{R}^n$. The {\em Hausdorff metric} between $A$ and $B$ is defined by
\begin{equation*}
d_H(A, B):=\max\bigg\{\sup_{a\in A}\inf_{b\in B}\|a-b\|, \sup_{b\in B}\inf_{a\in A}\|a-b\|\bigg\}.
\end{equation*}

Let $\{A_n\}$ and $A$ be closed and bounded intervals in $\mathbb{R}^n$. We say that the sequence $\{A_n\}$ converges to $A$, denoted by $$\lim_{n\to\infty}A_n=A,$$ 
if, for every $\varepsilon>0$, there exists $N\in \mathbb{N}$ such that, for every $n\geq N$, we have $d_H(A_n, A)<\varepsilon$.

We  recall here the definitions of the $LU$ interval order relation in $\mathcal{K}_c$ and the corresponding solution concepts of \eqref{problem}. 
\begin{definition}[{see \cite{Ishibuchi-Tanaka-90,Wu-07}}]{\rm 
		Let $A=[a^L, a^U]$ and $B=[b^L, b^U]$ be two intervals in $\mathcal{K}_c$. We say that:
		\begin{enumerate}[(i)]
			\item $A\preceq_{LU} B$ if $a^L\leq b^L$ and $a^U\leq b^U$.
			
			\item  $A\prec_{LU} B$ if $A\preceq_{LU} B$ and $A\neq B$, or, equivalently, $A\prec_{LU} B$ if
			\\
			$ 
			\begin{cases}
			a^L<b^L
			\\
			a^U\leq b^U,
			\end{cases}
			$ 
			or \ \ \ \
			$ 
			\begin{cases}
			a^L\leq b^L
			\\
			a^U< b^U,
			\end{cases}
			$ 
			or \ \ \ \
			$ 
			\begin{cases}
			a^L< b^L
			\\
			a^U< b^U.
			\end{cases}
			$ 
			\item $A\prec^s_{LU} B$ if $a^L<b^L$ and $a^U<b^U$.
		\end{enumerate}
	}
\end{definition}

\begin{definition}[see \cite{Wu-09}]{\rm  Let $x^*\in X$. We say that
		\begin{enumerate}[(i)]
			\item $x^*$ is an ${LU}$-solution of \eqref{problem}, if there is no $x\in X$ such that 
			$$f(x)\prec_{LU}f(x^*).$$

			\item $x^*$ is a weakly $LU$-solution of \eqref{problem}, if there is no $x\in X$ such that 
			$$f(x)\prec^s_{LU}f(x^*).$$ 
		\end{enumerate}
	}
\end{definition} 

The set of weakly $LU$-solutions and the set of $LU$-solutions of \eqref{problem} are denoted, respectively, by  $\mathcal{S}^w\eqref{problem}$   and $\mathcal{S}\eqref{problem}$. Clearly, 
\begin{equation*}
\mathcal{S}\eqref{problem}\subset \mathcal{S}^w\eqref{problem}.
\end{equation*}  

We now introduce approximate solutions of \eqref{problem} with respect to $LU$ interval order relation.  Let $\epsilon^L$ and  $\epsilon^U$ be two real numbers satisfying $0\leq \epsilon^L\leq \epsilon^U$ and put   $\mathcal{E}:=[\epsilon^L, \epsilon^U]\in \mathcal{K}_c$.
\begin{definition}{\rm
		Let $x^*\in X$. We say that:
		\begin{enumerate}[(i)]
			\item $x^*$ is an {\em $\mathcal{E}$-$LU$-solution} of \eqref{problem} if there is no $x\in X$ such that 
			$$f(x)\prec_{LU}f(x^*)-\mathcal{E}.$$  
			
			\item $x^*$ is a {\em weakly $\mathcal{E}$-$LU$-solution} of \eqref{problem} if there is no $x\in X$ such that 
			$$f(x)\prec^s_{LU}f(x^*)-\mathcal{E}.$$ 
			
			\item  $x^*$ is an {\em $\mathcal{E}$-quasi-$LU$-solution} of \eqref{problem} if there is no $x\in X$ such that 
			$$f(x)\prec_{LU}f(x^*)-\mathcal{E}\|x-x^*\|.$$  
			
			\item  $x^*$ is a {\em weakly $\mathcal{E}$-quasi-$LU$-solution} of \eqref{problem} if there is no $x\in X$ such that 
			$$f(x)\prec^s_{LU}f(x^*)-\mathcal{E}\|x-x^*\|.$$      
		\end{enumerate}
	}
\end{definition}
We denote the set of $\mathcal{E}$-$LU$-solutions (resp., weakly \mbox{$\mathcal{E}-LU$-solutions}, $\mathcal{E}$-quasi-$LU$-solution, weakly $\mathcal{E}$-quasi-$LU$-solutions)  of \eqref{problem} by $\mathcal{E}$-$\mathcal{S}$\eqref{problem} (resp., $\mathcal{E}$-$\mathcal{S}^w$\eqref{problem}, $\mathcal{E}$-quasi-$\mathcal{S}$\eqref{problem}, $\mathcal{E}$-quasi-$\mathcal{S}^w$\eqref{problem}). Clearly,

\centerline{  $\mathcal{E}$-$\mathcal{S}$\eqref{problem}$\subset$ $\mathcal{E}$-$\mathcal{S}^w$\eqref{problem} and $\mathcal{E}$-quasi-$\mathcal{S}$\eqref{problem} $\subset$ $\mathcal{E}$-quasi-$\mathcal{S}^w$\eqref{problem}.}  

It is easily seen that, when $\mathcal{E}=0$, i.e., $\epsilon^L=\epsilon^U=0$, then the notions of  an $\mathcal{E}$-$LU$-solution and an $\mathcal{E}$-quasi-$LU$-solution (resp., a weakly $\mathcal{E}$-$LU$-solution and a weakly $\mathcal{E}$-quasi-$LU$-solution) defined above coincide with the  one of an $LU$-solution (resp., a weakly $LU$-solution).     

The new concepts of approximate solutions of \eqref{problem} are closed related to the approximate efficient solutions of multiobjective optimization problems. In order to present the relationship between these solution concepts, we first  recall   some types of approximate efficient solutions in multiobjective optimization. Consider the following multiobjective optimization problem
\begin{equation}\label{GM-problem}
\mathrm{Min}\,_{\mathbb{R}^k_+} \{\widetilde{F}(x)\,:\, x\in X\},\tag{$\widetilde{\mathrm{MP}}$} 
\end{equation}
where $\widetilde{F}\colon\mathbb{R}^n\to\mathbb{R}^k$ is a vector-valued function defined on $\mathbb{R}^n$. Let $\epsilon\in\mathbb{R}^k_+$ and $x^*\in X$. We say that:
\begin{enumerate}[(i)]
	\item $x^*$ is an  {\em $\epsilon$-efficient solution} of \eqref{GM-problem} if there is no $x\in X$ such that
	\begin{equation*}
	\widetilde{F}(x)\in \widetilde{F}(x^*)-\epsilon-\mathbb{R}^k_+\setminus\{0\}.
	\end{equation*}
	
	\item $x^*$ is a  {\em weakly $\epsilon$-efficient solution} of \eqref{GM-problem} if there is no $x\in X$ such that
	\begin{equation*}
	\widetilde{F}(x)\in \widetilde{F}(x^*)-\epsilon-\mathrm{int}\,\mathbb{R}^k_+.
	\end{equation*}
	
	\item $x^*$ is an  {\em $\epsilon$-quasi-efficient solution} of \eqref{GM-problem} if there is no $x\in X$ such that
	\begin{equation*}
	\widetilde{F}(x)\in \widetilde{F}(x^*)-\epsilon\|x-x^*\|-\mathbb{R}^k_+\setminus\{0\}.
	\end{equation*}
	
	\item $x^*$ is a {\em weakly  $\epsilon$-quasi-efficient solution} of \eqref{GM-problem} if there is no $x\in X$ such that
	\begin{equation*}
	\widetilde{F}(x)\in \widetilde{F}(x^*)-\epsilon\|x-x^*\|-\mathrm{int}\,\mathbb{R}^k_+.
	\end{equation*}
\end{enumerate}

\begin{lemma}\label{Lemma-2.1} Let  $\epsilon:=(\epsilon^U, \epsilon^L)$.  A point $x^*$ is an $\mathcal{E}$-$LU$-solution of \eqref{problem} if and only if $x^*$ is an $\epsilon$-efficient solution of the following multiobjective optimization problem:
	\begin{equation}\label{M-problem}
	\mathrm{Min}\,_{\mathbb{R}^2_+}\{F(x)\,:\, x\in X\},\tag{MP}
	\end{equation}
	where $F(x):=(f^L(x), f^U(x))$ for all $x\in\mathbb{R}^n$.
\end{lemma}
\begin{proof} Let $x^*$ be an $\mathcal{E}$-$LU$-solution of \eqref{problem}. Then, there is no $x\in X$ satisfying
	\begin{equation}\label{equa-2.1}
	f(x)\prec_{LU}f(x^*)-\mathcal{E}.
	\end{equation}
	We claim that $x^*$ is an $\epsilon$-efficient solution  of  \eqref{M-problem}. Indeed,  if otherwise, then there exists $\bar x\in X$ such that
	\begin{equation*}
	F(\bar x)\in F(x^*)-\epsilon-\mathbb{R}^2_+\setminus\{0\},
	\end{equation*}   
	or, equivalently,
	\begin{equation*}
	\begin{cases}
	f^L(\bar x)&\leq f^L(x^*)-\epsilon^U,
	\\
	f^U(\bar x)&\leq f^U(x^*)-\epsilon^L, 
	\end{cases}
	\end{equation*}
	with at least one strict inequality. Hence, $f(\bar x)\prec_{LU}f(x^*)-\mathcal{E}$,  which contradicts to \eqref{equa-2.1}.
	
	Conversely, let $x^*$ be an $\epsilon$-efficient solution of \eqref{M-problem}.  Then, there is no $x\in X$ such that
	\begin{equation*} 
	F(x)\in F(x^*)-\epsilon-\mathbb{R}^2_+\setminus\{0\}.
	\end{equation*}
	This means that there is no $x\in X$ satisfying
	\begin{equation*}
	\begin{cases}
	f^L(x)&\leq f^L(x^*)-\epsilon^U,
	\\
	f^U(x)&\leq f^U(x^*)-\epsilon^L,
	\end{cases}
	\end{equation*}
	with at least one strict inequality.  This implies that
	\begin{equation*}
	f(x)\prec_{LU} f(x^*)-\mathcal{E}.
	\end{equation*}
	Hence, $x^*$ is an $\mathcal{E}$-$LU$-solution of \eqref{problem}. 
\end{proof}

\begin{lemma}\label{Lemma-2.2} Let  $\epsilon:=(\epsilon^U, \epsilon^L)$. A point $x^*$ is a weakly $\mathcal{E}$-$LU$-solution (resp., an $\mathcal{E}$-quasi-$LU$-solution, a weakly $\mathcal{E}$-quasi-$LU$-solution) of \eqref{problem} if and only if $x^*$ is a weakly $\epsilon$-efficient solution (resp., an \mbox{$\epsilon$-quasi-efficient solution}, a weakly $\epsilon$-quasi-efficient solution) of  \eqref{M-problem}.
\end{lemma}
\begin{proof}
	The proof is quiet similar to that of the proof of Lemma \ref{Lemma-2.1}, so omitted.
\end{proof}
\section{EXISTENCE THEOREMS}\label{Existence theorems}
In this section, we assume that $X$ is a nonempty and closed subset in $\mathbb{R}^n$.
\begin{definition}{\rm  We say that the function $f$ is {\em $LU$-bounded from below on} $X$ if there exists an interval $B=[b^L, b^U]\in \mathcal{K}_c$ such that
		\begin{equation*}
		B\preceq_{LU} f(x), \ \ \forall x\in X. 
		\end{equation*}
	}
\end{definition}

By definition, it is easily seen that the interval-valued function $f$ is $LU$-bounded from below on $X$ if and only if the function  $f^L$ is bounded from below on $X$.
\begin{theorem}[Existence of  $\mathcal{E}$-$LU$-solutions] \label{Theorem-3.1} 
	Assume that $f$ is $LU$-bounded from below on $X$ by an interval $B\in\mathcal{K}_c$. Then, for each $\mathcal{E}=[\epsilon^L, \epsilon^U]\in\mathcal{K}_c$ satisfying $0\prec_{LU}\mathcal{E}$,   the problem \eqref{problem} admits at least one $\mathcal{E}$-$LU$-solution.
\end{theorem}
\begin{proof} Let $x^0\in X$ and put 
	$$[f(X)]_{f(x^0)}:=\{A\in f(X)\,:\, A\preceq_{LU} f(x^0)\},$$
	where $f(X):=\{f(x)\,:\, x\in X\}$. 
	We first claim that there exists a point $x^*\in f^{-1} ([f(X)]_{f(x^0)} )$ such that 
	\begin{equation}\label{equ-3.1}
	f(x)\nprec_{LU} f(x^*)-\mathcal{E}, \ \ \forall x\in   f^{-1}\big([f(X)]_{f(x^0)}\big),
	\end{equation}
	where $f^{-1} ([f(X)]_{f(x^0)} ):=\{x\in X\,:\, f(x)\in [f(X)]_{f(x^0)}\}$. Indeed, if such a point $x^*$  does not exist, we can find a sequence $\{x^k\}\subset f^{-1} ([f(X)]_{f(x^0)})$ such that
	\begin{equation*}
	f(x^k)\prec_{LU} f(x^{k-1})-\mathcal{E}, \ \ \forall k\in\mathbb{N}.
	\end{equation*}
	Summarizing these inequalities up to $k$, we obtain
	\begin{equation*}
	f(x^k)\prec_{LU} f(x^0)-k\mathcal{E}, \ \ \forall k\in\mathbb{N},
	\end{equation*}
	or, equivalently,
	\begin{equation}\label{equa-3.2}
	\frac{1}{k}f(x^k)\prec_{LU} \frac{1}{k}f(x^0)-\mathcal{E}, \ \ \forall k\in\mathbb{N}.
	\end{equation}
	Due to the construction of the sequence $\{x^k\}$  and the the $LU$-boundedness from below on $X$ of $f$, we have
	$$B\preceq_{LU} f(x^k)\prec_{LU} f(x^0), \ \ \forall k\in\mathbb{N}.$$
	Hence, $\frac{1}{k}f(x^k)\to 0$ as $k\to\infty$. Then letting $k\to\infty$ in \eqref{equa-3.2}, we obtain $\mathcal{E}\preceq_{LU} 0$, which contradicts to the fact that $0\prec_{LU}\mathcal{E}$.
	
	We now prove that   \eqref{equ-3.1} holds also for $x\in X\setminus f^{-1} ([f(X)]_{f(x^0)} )$. Indeed, if otherwise, then there exists $x\in X\setminus f^{-1} ([f(X)]_{f(x^0)} )$ such that
	\begin{equation*} 
	f(x)\prec_{LU} f(x^*)-\mathcal{E}.
	\end{equation*} 
	Combining this with the fact that $f(x^*)\preceq_{LU} f(x^0)$, we obtain $f(x)\prec_{LU} f(x^0)$, a contradiction. Therefore,
	\begin{equation*} 
	f(x)\nprec_{LU} f(x^*)-\mathcal{E}, \ \ \forall x\in   X.
	\end{equation*}
	This means that $x^*$ is an $\mathcal{E}$-$LU$-solution of \eqref{problem}. 
\end{proof}

\begin{example}\label{Example-3.1}\rm 
	Let $f\colon\mathbb{R}^2\to\mathcal{K}_c$ be an interval-valued function defined by 
	$$f(x)=[f^L(x), f^U(x)]=[x_1^2+(x_1x_2-1)^2, 2x_1^2+(x_1x_2-1)^2]$$ 
	for all $x=(x_1, x_2)\in \mathbb{R}^2$ and  let $X=\mathbb{R}^2$. Then we have
	$0<f^L(x)\leq f^U(x)$ for all $x\in\mathbb{R}^2$. Hence, $f$ is bounded from below on $X$. We claim that  the problem \eqref{problem} has no weakly $LU$-solution. Indeed, let $x^*$ be an arbitrary point in $X$. Then, $0<f^L(x^*)\leq f^U(x^*)$. Let $\{x^k\}$ be a sequence defined by $x^k=(\frac{1}{k}, k)$ for each $k\in \mathbb{N}$. Then we have
	\begin{align*}
	&\lim\limits_{k\to\infty}f^L(x^k)=\lim\limits_{k\to\infty}\frac{1}{k^2}=0<f^L(x^*),
	\\
	&\lim\limits_{k\to\infty}f^U(x^k)=\lim\limits_{k\to\infty}\frac{2}{k^2}=0<f^U(x^*).
	\end{align*}
	This implies that there exists $K\in\mathbb{N}$ such that
	\begin{align*}
	f^L(x^k)<f^L(x^*),
	\\
	f^U(x^k)<f^U(x^*),
	\end{align*}
	for all $k\geq K$. Hence, $x^*$ is not a weakly $LU$-solution of \eqref{problem}. This means that the set $\mathcal{S}^w\eqref{problem}$ is empty. Consequently, $\mathcal{S}\eqref{problem}$ is empty. 
	
	However, by Theorem \ref{Theorem-3.1}, for all $\mathcal{E}\in\mathcal{K}_c$, $0\prec_{LU}\mathcal{E}$,    \eqref{problem} has at least an $\mathcal{E}$-$LU$-solution. Consequently, \eqref{problem} admits at least one weakly $\mathcal{E}$-$LU$-solution. 
\end{example}

We say that the function $f$ is {\em lower-semicontinuous} if $f^L$ and $f^U$ are   lower-semicontinuous functions.
\begin{theorem}[Existence of $\mathcal{E}$-quasi-$LU$-solutions]\label{Existence-quasi-solution}
	Assume that $f$ is lower-semicontinuous and $LU$-bounded from below on $X$ by an interval $B\in\mathcal{K}_c$. Then, for every $0\prec^s_{LU}\mathcal{E}$, the problem \eqref{problem} admits at least one $\mathcal{E}$-quasi-$LU$-solution.
\end{theorem}
To prove Theorem \ref{Existence-quasi-solution}, we need the following vectorial Ekeland's variational principle. 
\begin{lemma}[{see \cite[Theorem 3.1]{araya}}] \label{Lemma-araya}
	Let $(X,d)$ be a complete metric space and $Y$ a Banach space. Assume that $C\subset Y$ is a closed, convex and pointed cone with $\mathrm{int}\,C\neq \emptyset$. Let $k^0\in\mathrm{int}\,C$ and let $F\colon X\to Y$ be a vector-valued function. For every $\varepsilon>0$ there is an initial point $x_0\in X$  such that $F(X)\cap (F(x_0)-\varepsilon k^0-\mathrm{int}\,C)=\emptyset$ and $F$ satisfies
	\begin{equation*}
	\{x^\prime\in X\,:\, F(x^\prime)+d(x^\prime, x)k^0\in F(x)-C\} \ \ \text{is closed for every}\ \ x\in X.
	\end{equation*}
	Then there exists $\bar x\in X$ such that
	\begin{enumerate}[\rm(i)]
		\item $F(\bar x)\in F(x_0)-\mathrm{int}\,C$,
		\item $d(x_0, \bar x)\leq 1$
		\item $F(x)\notin F(\bar x)-\varepsilon d(x, \bar x)k^0-C$ for all $x\neq \bar x$.
	\end{enumerate}
\end{lemma} 
\noindent{\em Proof of Theorem \ref{Existence-quasi-solution}.}
Let $x^0\in X$. Then, by Theorem \ref{Theorem-3.1}, the problem \eqref{problem} has at least an $\mathcal{E}$-$LU$-solution, say $x^*$,  satisfying  $f(x^*)\preceq_{LU}f(x^0)$. By Lemma \ref{Lemma-2.1}, $x^*$ is an $\epsilon$-efficient solution   of \eqref{M-problem} and so is a weakly $\epsilon$-efficient solution of \eqref{M-problem}, where $\epsilon=(\epsilon^U, \epsilon^L)$.   Consequently,
\begin{equation*}
F(X)\cap [F(x^*) - \epsilon - \mathrm{int}\,\mathbb{R}^2_+]\ =\  \emptyset.
\end{equation*}
By the lower-semicontinuity of $f^L, f^U$ and the closedness of $X$, it is easy to see that for each $x\in X$ the following set
$$\{u\in X\,:\, F(u)+\epsilon\|u-x\|\in F(x) -{\mathbb{R}^2}\}$$
is closed. By Lemma \ref{Lemma-araya}, there exists a point $\bar x\in X$ such that $F(\bar x)\in F(x^*) -\mathrm{int}\,\mathbb{R}^2$ and
$$F(x)\notin F(\bar x)-\epsilon\|x-\bar x\|-\mathbb{R}^2_+, \ \ \forall x\in X\setminus\{\bar x\}.$$
Hence,
$$F(x)\notin F(\bar x)-\epsilon\|x-\bar x\|-\mathbb{R}^2_+\setminus\{0\}, \ \ \forall x\in X,$$
or, equivalently, $\bar x$ is an $\epsilon$-quasi-efficient solution of \eqref{M-problem}. Thus, $\bar x$ is an $\mathcal{E}$-quasi-$LU$-solution of \eqref{problem} due to Lemma \ref{Lemma-2.2}. The proof is complete. $\hfill\Box$   
\begin{example}\rm 
	Let $f$ and $X$ be as in Example \ref{Example-3.1}. Then, by Theorem \ref{Existence-quasi-solution},   the  sets $\mathcal{E}$-quasi-$\mathcal{S}\eqref{problem}$ and $\mathcal{E}$-quasi-$\mathcal{S}^w\eqref{problem}$  are nonempty for all $\mathcal{E}\in\mathcal{K}_c$ satisfying $0\prec^s_{LU}\mathcal{E}$. 
\end{example} 

\section{KKT OPTIMALITY CONDITIONS}\label{KKT conditions}
In this section, we assume that  $f^L$, $f^U$, and $g_j$, $j=1, \ldots, m$, are  convex functions. Since every real-valued convex function is continuous, the constraint set $X$ is closed and convex. In order to present optimality conditions for approximate solutions of \eqref{problem}, we  recall some notations and basic results from convex analysis.

\subsection{The approximate subdifferential}

Let $\varphi\colon\mathbb{R}^n\to \overline{\mathbb{R}}$ be a convex function. The {\em conjugate function} of $\varphi$, $\varphi^*\colon\mathbb{R}^n\to \overline{\mathbb{R}}$, is defined by
\begin{equation*}
\varphi^*(v):=\sup\{\langle v, x\rangle-\varphi(x)\,:\, x\in\mathrm{dom}\,\varphi\}.
\end{equation*} 
For  $\varepsilon\geq 0$ the {\em  $\varepsilon$-subdifferential } of $\varphi$ at $x^*\in \mathrm{dom}\,\varphi$ is given by:
\begin{equation*}
\partial_{\varepsilon}\varphi (x^*):=\{v\in\mathbb{R}^n\,:\, \varphi(x)-\varphi(x^*)\geq \langle v, x-x^*\rangle-\varepsilon,  \ \ \forall x\in\mathrm{dom}\,\varphi\}.
\end{equation*}
When $\varepsilon=0$, $\partial_0\varphi(x^*)$ coincides with $\partial\varphi(x^*)$, the subdifferential of $\varphi$ at $x^*$  (see, e.g., \cite{Rockafellar-70,Draha-Dutta-12}). It is well-known that
\begin{equation*}
\partial_{\mu\varepsilon}(\mu\varphi (x^*))=\mu\partial_{\varepsilon}\varphi(x^*), \ \ \forall \mu>0.
\end{equation*}
\begin{lemma}[{Sum rule \cite[Theorem 2.115]{Draha-Dutta-12}}]
	Consider two proper convex functions $\varphi_i\colon\mathbb{R}^n\to\overline{\mathbb{R}}$, $i=1, 2$, such that $\mathrm{ri}\,\mathrm{dom}\varphi_1\cap \mathrm{ri}\,\mathrm{dom}\varphi_2\neq\emptyset$. Then for $\varepsilon\geq 0$ and $x\in\mathrm{dom}\varphi_1\cap \mathrm{dom}\varphi_2$,
	\begin{equation*}
	\partial_{\varepsilon}(\varphi_1+\varphi_2)(x)=\bigcup_{\substack{\varepsilon_1+\varepsilon_2=\varepsilon\\\varepsilon_1,\;\varepsilon_2\geq0}}\big(
	\partial_{\varepsilon_1}\varphi_1(x)+\partial_{\varepsilon_2}\varphi_2(x)\big).
	\end{equation*}
\end{lemma}

We say that the constrain set $X$ satisfies the {\em Slater constraint qualification} if there exists $\hat{x}\in \mathbb{R}^n$ such that $g_j(\hat{x})<0$, for all $j\in J$. The following result gives  necessary and sufficient optimality conditions for a feasible point to be an approximate solution of a convex programming problem. 

\begin{theorem}[{See \cite[Theorem 10.9]{Draha-Dutta-12}}]\label{Theorem-4.1} Let $\varphi\colon\mathbb{R}^n\to\mathbb{R}$ be a convex function  and let $\varepsilon\geq 0$. Assume that $X$ satisfies the Slater constraint qualification. Then $x^*\in X$ is a $\varepsilon$-solution of $\varphi$ on $X$, i.e., $\varphi(x)\geq \varphi(x^*)-\varepsilon$ for all $x\in X$, if and only if there exist $\varepsilon_0\geq 0$, $\varepsilon_j\geq 0$, and $\lambda_j\geq 0$, $j\in J$, such that
	\begin{equation*}
	0\in \partial_{\varepsilon_0}\varphi(x^*)+\sum_{j=1}^{m}\partial_{\varepsilon_j}(\lambda_jg_j)(x^*)\ \ \text{and}\ \ \sum_{j=0}^{m}\varepsilon_j-\varepsilon\leq \sum_{j=1}^{m}\lambda_jg_j(x^*).
	\end{equation*} 
	
\end{theorem}

\subsection{KKT conditions for weakly $\mathcal{E}$-$LU$-solutions} 
\begin{lemma} Let $x^*\in X$.  Then $x^*$ is a weakly-$\mathcal{E}$-$LU$-solution
	of \eqref{problem} if and only if there exist $\mu^L\geq 0$, $\mu^U\geq 0$, $\mu^L+\mu^U=1$ such that
	\begin{equation}\label{equ-lm4.2-1}
	\mu^L f^L(x)+\mu^Uf^U(x)\geq \mu^L f^L(x^*)+\mu^Uf^U(x^*)-\mu^L\epsilon^U-\mu^U\epsilon^L, \ \ \forall x\in X.
	\end{equation}
\end{lemma}
\begin{proof}
	By Lemma \ref{Lemma-2.2}, $x^*$ is a  weakly-$\mathcal{E}$-$LU$-solution
	of \eqref{problem} if and only if $x^*$ a weakly $\epsilon$-efficient solution \eqref{M-problem}, where $\epsilon:=(\epsilon^U, \epsilon^L)$. This is equivalent to the inconsistent of the following system
	\begin{equation*}
	\begin{cases}
	&f^L(x)< f^L(x^*)-\epsilon^U,
	\\
	&f^U(x)< f^U(x^*)-\epsilon^L,
	\\
	&x\in X.
	\end{cases}
	\end{equation*}
	By \cite[Theorem 1]{Fan57},   the above system is inconsistent if and only if   there exist $\mu^L\geq 0$, $\mu^U\geq 0$, $\mu^L+\mu^U=1$ such that 
	\begin{equation*}
	\mu^L[f^L(x)-f^L(x^*)+\epsilon^U]+\mu^U[f^U(x)-f^U(x^*)+\epsilon^L]\geq 0, \ \ \forall x\in X,
	\end{equation*} 
	or, equivalently, \eqref{equ-lm4.2-1}  is valid.  
\end{proof}

The following result gives KKT necessary and sufficient optimality conditions  for a feasible point to be a weakly-$\mathcal{E}$-$LU$-solution
of \eqref{problem}.
\begin{theorem}  Let $x^*\in X$. Assume that $X$ satisfies the Slater constraint qualification.  Then $x^*$ is a weakly-$\mathcal{E}$-$LU$-solution
	of \eqref{problem} if and only if there exist $\mu^L\geq 0$, $\mu^U\geq 0$, $\mu^L+\mu^U=1$, $\varepsilon_0\geq 0$,  $\varepsilon_j\geq 0$, and $\lambda_j\geq 0$, $j\in J$, such that
	\begin{equation*}
	0\in \partial_{\varepsilon_0}(\mu^Lf^L+\mu^Uf^U)(x^*)+ \sum_{j=1}^{m}\partial_{\varepsilon_j}(\lambda_jg_j)(x^*)\ \ \text{and}\ \ \sum_{j=0}^{m}\varepsilon_j-\mu^L\epsilon^U-\mu^U\epsilon^L\leq \sum_{j=1}^{m}\lambda_jg_j(x^*).
	\end{equation*}
\end{theorem}
\begin{proof}  
	The proof is directly from Lemma \ref{Lemma-2.2} and Theorem \ref{Theorem-4.1}, so omitted.
\end{proof}

\subsection{KKT conditions for $\mathcal{E}$-$LU$-solutions}
For each $x^*\in X$, denote  
\begin{equation*}
X(x^*, \mathcal{E}):=\{x\in \mathbb{R}^n\,:\, f(x)\preceq_{LU} f(x^*)-\mathcal{E}\}.
\end{equation*}
\begin{lemma}\label{Necessary-efficient} Let $x^*\in X$. Then $x^*$ is an $\mathcal{E}$-$LU$-solution
	of \eqref{problem} if and only if $X\cap X(x^*,\mathcal{E})=\emptyset$ or,  
	\begin{equation}\label{equ-43}
	f^L(x)+f^U(x)=  f^L(x^*)+ f^U(x^*)- \epsilon^U- \epsilon^L, \ \ \forall x\in X\cap X(x^*,\mathcal{E}).
	\end{equation}
\end{lemma}
\begin{proof} We will follow the proof scheme of \cite[Proposition 8.1]{Lee-Kim-Dinh-12} (see also \cite[Proposition 3.1]{Loridan-84}).
	
	$(\Rightarrow)$: Let $x^*$ be an  $\mathcal{E}$-$LU$-solution
	of \eqref{problem}. Then, by Lemma \ref{Lemma-2.1},  $x^*$ is an $\epsilon$-efficient solution of \eqref{M-problem}, where $\epsilon=(\epsilon^U, \epsilon^L)$. This means that  there is no $x\in X$ such that
	\begin{equation*}
	\begin{cases}
	f^L(x)&\leq f^L(x^*)-\epsilon^U
	\\
	f^U(x)&\leq f^U(x^*)-\epsilon^L
	\end{cases}
	\end{equation*} 
	with at least one strict inequality. Hence, $X(x^*, \mathcal{E})=\emptyset$ or,
	\begin{equation*}
	\begin{cases}
	f^L(x)&= f^L(x^*)-\epsilon^U
	\\
	f^U(x)&= f^U(x^*)-\epsilon^L
	\end{cases}
	\end{equation*}
	for all $x\in X\cap X(x^*,\mathcal{E})$ and we therefore get \eqref{equ-43}.  
	
	$(\Leftarrow)$: Clearly, if $X\cap X(x^*,\mathcal{E})=\emptyset$, then $x^*$ is an $\mathcal{E}$-$LU$-solution of \eqref{problem}. We now assume that \mbox{$X\cap X(x^*,\mathcal{E})\neq\emptyset$} and \eqref{equ-43} holds. Suppose to the contrary that $x^*$ is not an $\mathcal{E}$-$LU$-solution of \eqref{problem}. Then, there exists $x\in X$ such that
	\begin{equation*}
	\begin{cases}
	f^L(x)&\leq f^L(x^*)-\epsilon^U
	\\
	f^U(x)&\leq f^U(x^*)-\epsilon^L
	\end{cases}
	\end{equation*} 
	with at least one strict inequality. Hence, $x\in X\cap X(x^*, \mathcal{E})$ and  
	$$f^L(x)+f^U(x)<  f^L(x^*)+ f^U(x^*)- \epsilon^U- \epsilon^L,$$
	contrary to \eqref{equ-43}.   
\end{proof}

We say the {\em  closedness condition } $(CC_{x^*})$ holds at $x^*\in X$ if 
\begin{equation*}
\mathrm{cone}\,\Big(\bigcup_{j\in J} \mathrm{epi}\,g_j^*\cup \mathrm{epi}\, (\tilde{f^L})^* \cup \mathrm{epi}\, (\tilde{f^U})^*\Big)\ \ \text{is closed},
\end{equation*}  
where $\tilde{f^L}(x):=f^L(x)-f^L(x^*)+\epsilon^U$ and $\tilde{f^U}(x):=f^U(x)-f^U(x^*)+\epsilon^L$. 

By using Lemma \ref{Necessary-efficient} and  modifying  the proof of Theorem 8.1 in \cite{Lee-Kim-Dinh-12}, we can obtain the following result.
\begin{theorem}  Let $x^*\in X$. Assume that $X\cap X(x^*,\mathcal{E})\neq\emptyset$ and that $(CC_{x^*})$ holds.   Then $x^*$ is an $\mathcal{E}$-$LU$-solution
	of \eqref{problem} if and only if there exist  $\varepsilon_0\geq 0$, $\gamma_1\geq 0$, $\gamma_2\geq 0$,   $\varepsilon_j\geq 0$, $\mu_1\geq 0$, $\mu_2\geq 0$, and $\lambda_j\geq 0$, $j\in J$, such that
	\begin{align*}
	&0\in \partial_{\varepsilon_0}(f^L+f^U)(x^*)+\mu_1\partial_{\gamma_1} f^L(x^*)+\mu_2\partial_{\gamma_2} f^U(x^*)+ \sum_{j=1}^{m}\partial_{\varepsilon_j}(\lambda_jg_j)(x^*)
	\\ 
	&\varepsilon_0+\mu_1\gamma_1+\mu_2\gamma_2-(1+\mu_1)\epsilon^U-(1+\mu_2)\epsilon^L+\sum_{j=1}^{m}\lambda_j\varepsilon_j\leq \sum_{j=1}^{m}\lambda_jg_j(x^*).
	\end{align*}
\end{theorem}
\begin{proof}
	The proof is  similar to that of the proof of \cite[Theorem 8.1]{Lee-Kim-Dinh-12}, and we omit it.
\end{proof}
\subsection{KKT conditions for $\mathcal{E}$-quasi-$LU$-solutions}
We say that the {\em Mangasarian--Fromovitz constraint
	qualification} $(MFCQ)$ holds at $x^*\in X$ if there do not exist $\lambda_j\geq 0$, $j\in J(x^*)$ not all zero, such that
\begin{equation*}
0\notin \sum_{j\in J(x^*)}\lambda_j\partial g_j(x^*),
\end{equation*} 
where $J(x^*):=\{j\in J\,:\, g_j(x^*)=0\}$.
\begin{theorem} 
	Let $x^*\in X$. Assume that the condition $(MFCQ)$ holds at $x^*$.  Then $x^*$ is a weakly $\mathcal{E}$-quasi-$LU$-solution of \eqref{problem} if and only if   there exist $\mu^L\geq 0$, $\mu^U\geq 0$, $\mu^L+\mu^U>0$, $\varepsilon_0\geq 0$,  $\varepsilon_j\geq 0$, and $\lambda_j\geq 0$, $j \in J$, such that 
	\begin{align} \label{sufficient-quasi}
	\begin{split}
	&0\in \mu^L\partial f^L(x^*)+\mu^U\partial f^U(x^*)+\sum_{j=1}^{m}\lambda_j\partial g_j(x^*)+(\mu^L\epsilon^U+\mu^U\epsilon^L)B_{\mathbb{R}^n}.
	\\
	&\lambda_jg_j(x^*)=0, \ \ \forall j\in J.
	\end{split}
	\end{align}
	Furthermore, if $f^L$ and $f^U$ are strictly convex, then \eqref{sufficient-quasi} is also sufficient for $x^*$ is an $\mathcal{E}$-quasi-$LU$-solution of \eqref{problem}. 
\end{theorem}
\begin{proof}
	$(\Rightarrow)$: Assume that $x^*$ is a weakly $\mathcal{E}$-quasi-$LU$-solution of \eqref{problem}. Then, by Lemma \ref{Lemma-2.2}, $x^*$ is a weakly $\epsilon$-quasi-efficient solution of \eqref{M-problem}, where $\epsilon=(\epsilon^U, \epsilon^L)$. Hence, the following system:
	\begin{equation*}
	\begin{cases}
	f^L(x)<f^L(x^*)-\epsilon^U\|x-x^*\|,
	\\
	f^U(x)<f^U(x^*)-\epsilon^L\|x-x^*\|,
	\end{cases}
	\end{equation*}
	has no solution $x\in X$. For each $x\in \mathbb{R}^n$, put
	\begin{equation*}
	\Phi(x):=\max\{f^L(x)-f^L(x^*)+\epsilon^U\|x-x^*\|, f^U(x)-f^U(x^*)+\epsilon^L\|x-x^*\|, g_1(x),  \ldots, g_m(x)\}.
	\end{equation*}  
	Then $\Phi(x^*)=0$ and  $\Phi(x)\geq 0$ for all $x\in \mathbb{R}^n$. Clearly, $\Phi$ is convex. Hence, by \cite[Theorem 2.89]{Draha-Dutta-12}, we have $$0\in\partial\Phi(x^*).$$ 
	From this and \cite[Theorem 2.96]{Draha-Dutta-12} it follows that there exist $\mu^L\geq 0$, $\mu^U\geq 0$, $\lambda_j\geq 0$, $j\in J(x^*)$ such that $\mu^L+\mu^U+\sum_{j\in J(x^*)}\lambda_j=1$ and
	\small{
		\begin{align*}
		0\in \mu^L\partial(f^L(\cdot)-f^L(x^*)+\epsilon^U\|\cdot\,-x^*\|)(x^*) +\mu^U\partial(f^U(\cdot)-f^U(x^*)+\epsilon^L\|\cdot\,-x^*\|)(x^*)
		+\sum_{j\in J(x^*)}\lambda_j\partial g_j(x^*).
		\end{align*}}
	Combining this with the Moreau--Rockafellar Sum Rule \cite[Theorem 2.91]{Draha-Dutta-12} and the fact that  $$\partial(\|\cdot\,-x^*\|)(x^*)=B_{\mathbb{R}^n},$$ 
	we obtain
	$$0\in \mu^L\partial f^L(x^*)+\mu^U\partial f^U(x^*)+\sum_{j\in J(x^*)}\lambda_j\partial g_j(x^*)+(\mu^L\epsilon^U+\mu^U\epsilon^L)B_{\mathbb{R}^n}.$$
	Clearly, the condition $(MFCQ)$ implies that $\mu^L+\mu^U>0$. For $j\notin J(x^*)$, we put $\lambda_j=0$. Then, \eqref{sufficient-quasi} holds. 
	
	$(\Leftarrow)$: Assume that there exist $\mu^L\geq 0$, $\mu^U\geq 0$, $\mu^L+\mu^U>0$, $\varepsilon_0\geq 0$,  $\varepsilon_j\geq 0$, and $\lambda_j\geq 0$, $j\in J$, that satisfy \eqref{sufficient-quasi}.  Hence, there exist $z^L\in\partial f^L(x^*)$,  $z^U\in\partial f^U(x^*)$, $u_j\in \partial g_j(x^*)$ and $b\in B_{\mathbb{R}^n}$ such that
	\begin{equation*}
	\mu^L z^L+\mu^U z^U+\sum_{j=1}^{m}\lambda_j u_j+(\mu^L\epsilon^U+\mu^U\epsilon^L)b=0,
	\end{equation*}
	or, equivalently,
	\begin{equation}\label{equ-4.5}
	\mu^L (z^L+\epsilon^Ub)+\mu^U  (z^U+\epsilon^Lb)+\sum_{j=1}^{m}\lambda_j u_j=0.
	\end{equation}
	Suppose to the contrary that $x^*$ is not a weakly $\mathcal{E}$-quasi-$LU$-solution of \eqref{problem}. This implies that there exists $\hat{x}\in X$ such that
	\begin{equation*}
	\begin{cases}
	f^L(\hat{x})<f^L(x^*)-\epsilon^U\|\hat{x}-x^*\|,
	\\
	f^U(\hat{x})<f^U(x^*)-\epsilon^L\|\hat{x}-x^*\|.
	\end{cases}
	\end{equation*}
	Hence,
	\begin{equation}\label{equ-4.4}
	\mu^L (f^L(\hat{x})+\epsilon^U\|\hat{x}-x^*\|-f^L(x^*))+\mu^U (f^U(\hat{x})+\epsilon^L\|\hat{x}-x^*\|-f^U(x^*))<0, 
	\end{equation}
	due to $\mu^L+\mu^U>0$. Since the function  $f^L(\cdot)+\epsilon^U\|\cdot\,-x^*\|$     is convex, we have
	\begin{align*}
	f^L(\hat{x})+\epsilon^U\|\hat{x}-x^*\|- f^L(x^*)\geq z^*(\hat{x}-x^*), \ \ \forall z^*\in \partial( f^L(\cdot)+\epsilon^U\|\cdot\,-x^*\|)(x^*). 
	\end{align*}
	This and the fact that 
	$$\partial( f^L(\cdot)+\epsilon^U\|\cdot\,-x^*\|)(x^*)= \partial  f^L(x^*)+\epsilon^U B_{\mathbb{R}^n}$$ 
	imply that
	\begin{equation*}
	f^L(\hat{x})+\epsilon^U\|\hat{x}-x^*\|- f^L(x^*)\geq (z^L+\epsilon^Ub)(\hat{x}-x^*).
	\end{equation*}
	Similarly, we have
	\begin{align*}
	f^U(\hat{x})+\epsilon^L\|\hat{x}-x^*\|- f^U(x^*)&\geq (z^U+\epsilon^Lb)(\hat{x}-x^*)
	\\
	g_j(\hat{x})-g_j(x^*)&\geq u_j(\hat{x}-x^*), \ \ \forall j\in J. 
	\end{align*}
	From these and \eqref{equ-4.4} we deduce that
	\begin{align*}
	\Big[\mu^L (z^L+\epsilon^Ub)+\mu^U  (z^U+\epsilon^Lb)+\sum_{j=1}^{m}\lambda_j u_j\Big](\hat{x}-x^*)
	\leq \mu^L (f^L(\hat{x})&+\epsilon^U\|\hat{x}-x^*\|-f^L(x^*))
	\\
	&+\mu^U (f^U(\hat{x})+\epsilon^L\|\hat{x}-x^*\|-f^U(x^*))<0,
	\end{align*}
	contrary to \eqref{equ-4.5}.
	
	If $f^L$ and $f^U$ are strictly convex, then  so are $f^L(\cdot)+\epsilon^U\|\cdot\,-x^*\|$ and $f^U(\cdot)+\epsilon^L\|\cdot\,-x^*\|$. Now suppose to the contrary that $x^*$ is not an $\mathcal{E}$-quasi-$LU$-solution of \eqref{problem}. Then there exists $\tilde{x}\in X$ such that
	\begin{equation*}
	\begin{cases}
	f^L(\tilde{x})\leq f^L(x^*)-\epsilon^U\|\tilde{x}-x^*\|,
	\\
	f^U(\tilde{x})\leq f^U(x^*)-\epsilon^L\|\tilde{x}-x^*\|,
	\end{cases}
	\end{equation*}
	with at least one strict inequality. Without loss of generality we assume that 
	\begin{equation*}
	f^L(\tilde{x})< f^L(x^*)-\epsilon^U\|\tilde{x}-x^*\|.
	\end{equation*}
	This imply that $\tilde{x}\neq x^*$. By the strictly convexity of $f^L$ and $f^U$ we obtain
	\begin{align*}
	f^L(\tilde{x})+\epsilon^U\|\tilde{x}-x^*\|- f^L(x^*)&> (z^L+\epsilon^Ub)(\tilde{x}-x^*),
	\\
	f^U(\tilde{x})+\epsilon^L\|\tilde{x}-x^*\|- f^U(x^*)&> (z^U+\epsilon^Lb)(\tilde{x}-x^*).
	\end{align*}
	Hence,
	\begin{align*}
	0=\Big[\mu^L (z^L+\epsilon^Ub)+\mu^U  (z^U+\epsilon^Lb)+\sum_{j=1}^{m}\lambda_j u_j\Big](\tilde{x}-x^*)
	< \mu^L (f^L&(\tilde{x})+\epsilon^U\|\tilde{x}-x^*\|-f^L(x^*))
	\\
	&+\mu^U (f^U(\tilde{x})+\epsilon^L\|\tilde{x}-x^*\|-f^U(x^*))\leq 0,
	\end{align*}
	a contradiction. The proof is complete.
\end{proof}
\section{Conclusions}\label{Conclusions} In this work, we propose some kinds of approximate solutions of interval optimization problems with respect to $LU$ interval order relation. We show that these approximate solutions exist under very weak assumptions. By establishing the relationships of approximate solutions  between interval optimization problems and multiobjective optimization problems and using suitable constraint qualifications, we derive some KKT necessary and sufficient  optimality conditions for approximate solutions of convex interval-valued optimization problems.  As shown in Lemmas \ref{Lemma-2.1} and \ref{Lemma-2.2},  approximate solutions of interval optimization problems are closed related to the approximate efficient ones of  multiobjective optimization problems. Accordingly, we may use the schemes in \cite{Chuong-Kim-16,Son-Tuyen-Wen-19,Tammer} to  present new results on KKT necessary and sufficient  optimality conditions by virtue of the Clarke subdifferentials (or the limiting  subdifferentials) for nonconvex and nonsmooth interval optimization problems in our further research.

\section*{Acknowledgments}{}
This research is funded by Hanoi Pedagogical University 2. 


\end{document}